\documentclass[12pt,letterpaper]{amsart}
\usepackage[utf8]{inputenc}
\usepackage[margin=1in]{geometry}
\usepackage{hyperref}
\usepackage{float}
\usepackage{tikz-cd} 
\usepackage{pst-all}
\usepackage{graphicx}
\usepackage{subcaption}
\usepackage{pstricks-add}
\usepackage{multicol}
\usepackage{colonequals}
\usepackage{graphicx}
\usepackage{yfonts}
\usepackage{amsthm}
\usepackage{amsmath}
\usepackage{amssymb}
\usepackage{url}
\usepackage{esvect}
\usepackage{tikz}
\DeclareMathOperator{\tri}{tri}
\DeclareMathOperator{\can}{can}
\DeclareMathOperator{\Ima}{Im}

\DeclareMathOperator{\SO}{SO}
\newcommand{\RR}{\mathbb{R}}
\usepackage[colorinlistoftodos]{todonotes}
\usepackage{enumerate}
\newtheorem{Theorem}{Theorem}
\newtheorem{Lemma}[Theorem]{Lemma}

\theoremstyle{definition}
\newtheorem{Definition}[Theorem]{Definition}

\theoremstyle{remark}
\newtheorem{Notation}[Theorem]{Notation}

\newcommand{\set}[2]{\left\{  #1  \ \middle| \ #2 \right\}  }
\renewcommand{\epsilon}{\varepsilon}

\numberwithin{Theorem}{section}
\numberwithin{equation}{section}

\begin{document}

        \title[Triangles Inscribed in Curves]{Inscribed triangles of Jordan curves in $\mathbb{R}^{n}$}
        \author{Aryaman Gupta}
        \address{Euler Circle, Mountain View, CA 94040}
        \email{aryamanjgupta@gmail.com}
        \author{Simon Rubinstein-Salzedo}
        \address{Euler Circle, Mountain View, CA 94040}
        \email{simon@eulercircle.com}
        \maketitle
        
    \begin{abstract}
           Nielsen's theorem states that any triangle can be inscribed in a planar Jordan curve. We prove a generalisation of this theorem, extending to any Jordan curve $J$ embedded in $\mathbb{R}^{n}$, for a restricted set of triangles. We then conclude by investigating a condition under which a given point of $J$ inscribes an equilateral triangle in particular.
    \end{abstract}
    
    \section{Introduction}
    
    A \emph{Jordan curve} is a continuous image of the unit interval in $\mathbb{R}^{n}$ that is injective everywhere except the endpoints, which are mapped to the same point. A polygon is \emph{inscribed} in a Jordan curve if the vertices of the polygon lie on the curve. There has been a considerable amount of interest surrounding the inscription of triangles and quadrilaterals in Jordan curves embedded in the plane. A lot of this interest stems from the Toeplitz square peg conjecture, which asks whether any Jordan curve in the plane has an inscribed square. Detailed exposition on this and similar problems can be found in~\cite{Matschke2014ASO} and~\cite{pak2010lectures}.

   In the literature, some of these variants have already been resolved. See, for example, ~\cite{meyerson1981balancing} and \cite{Nielsen1992}, wherein it is shown, respectively, that a planar Jordan curve necessarily inscribes a rectangle and any particular triangle. Yet, the original conjecture itself remains unproven, except under certain geometric or topological conditions. See, for instance,~\cite{stromquist_1989}, ~\cite{Matschke2014ASO},~\cite{10.2307/2370541} and \cite{Nielsen1995}.

    Here, instead of squares, we shall consider the inscription of triangles, under geometric conditions on the curve. The motivation for this paper comes from the following two results, proven respectively in in~\cite{Nielsen1992} and~\cite{Meyerson1980}.
    
    \begin{Theorem}[Nielsen] \label{Mot}
    Let $J \subset \mathbb{R}^{2}$ be a Jordan curve and let $\triangle$ be any triangle. Then infinitely many triangles similar to $\triangle$ can be inscribed in $J$.
    \end{Theorem}
    
    \begin{Theorem}[Meyerson] \label{Mey}
    Let $J \subset \mathbb{R}^{2}$ be a Jordan curve. For every point $p \in J$ except at most two, there exists an inscribed equilateral triangle such that one of its vertices is $p$.
    \end{Theorem}

    There has also been a smaller, but still significant, amount of interest in higher dimensional variants; see for instance~\cite{stromquist_1989},~\cite{article},~\cite{karasev2009inscribing}, and~\cite{Nielsen1995}. There are at least two difficulties in dealing with Jordan curves in at least three dimensions. Firstly, many proofs regarding inscription for planar Jordan curves (including the proof of Theorem~\ref{Mot}) rely upon the property that $J$ divides the plane into two disconnected subsets. Since this does not, of course, generalise to higher dimensions, these proofs cannot be generalised in any obvious way. Secondly, since Jordan curves in higher dimensions are able to form knots, they can potentially be much more pathological than planar Jordan curves.

    The aim of this paper is to prove that, subject to a certain geometric restriction, any triangle can be inscribed in a given Jordan curve $J$ embedded in $\mathbb{R}^{n}$. 
    
    Before stating the main result, we introduce some notation.
    \begin{itemize}
    \item $J$ denotes a Jordan curve embedded in $\mathbb{R}^{n}$, defined by $\gamma: [0,1] \rightarrow \mathbb{R}^{n}$.
    
    \item Let $\delta \in (0,\frac{1}{2})$. Then, $\Theta_{\delta}: (0,\delta) \times (0,\delta) \rightarrow \mathbb{R}_{\geq0}$ denotes the function that maps each pair $(s,s') \in (0,\delta) \times (0,\delta)$ to the angle between $\vv{o\gamma(s)}$  and $\overrightarrow{o\gamma(s')}$.

    \item $\Theta'_{\delta}: (1 - \delta,1) \times (0,\delta) \rightarrow \mathbb{R}_{\geq0}$ denotes the function that maps each pair $(s,s') \in (1 - \delta,1) \times (0,\delta)$ to the angle between $\overrightarrow{o\gamma(s)}$ and $\overrightarrow{o\gamma(s')}$.

    \item $\triangle abc$ denotes the triangle in $\mathbb{R}^{n}$ with vertices $a,b,c \in J$.
    \end{itemize}

    Here is our main result.
    \begin{Theorem}\label{The1}
     Let $\theta_{v}$ be the angle of some vertex $v$ of $\triangle$. If there exists a $\theta_{v}$ such that \[{\limsup_{\delta\rightarrow0^{+}}\Theta_{\delta} < \theta_{v} < \liminf_{\delta\rightarrow0^{+}}\Theta'_{\delta}},\] then there exist two points $p,q \in J\setminus\{o\}$ such that $\triangle opq$ is similar to $\triangle$, with vertex $v$ corresponding to $o$.
    \end{Theorem}
    
    Note that the hypothesis of Theorem~\ref{The1} is automatically satisfied when $o$ is a smooth point.
    
    For the proof, our first concern is how we identify when $o,p,q \in J$ inscribe $\triangle$. Assuming $\triangle opq$ is similar to $\triangle$, let $r \colonequals \frac{\| q - o \|}{\| p - o \|}$ and $r' \colonequals \frac{\| q - p \|}{\| p - o \|}$, where $r \geq 1$ without loss of generality.

    We fix $o$ to be $\gamma(0)$ and assume that it is the origin of $\mathbb{R}^{n}$, and we let $p$ be any element of $J - \{o\}$. Then, the set $\tri(o,p)$ of points $q$ such that $\triangle opq$ is similar to $\triangle$ is an $(n-2)$-sphere. Formally,
    \[\tri(o, p) \colonequals \left\{q \in \mathbb{R}^{n}\,\middle|\,\frac{\| \overrightarrow{Q} \|}{\|\overrightarrow{P} \|} = r\right\} \cap \left\{q \in \mathbb{R}^{n}\,\middle|\,\frac{\| \overrightarrow{Q} - \overrightarrow{P} \|}{\| \overrightarrow{P} \|} = r'\right\},\]
    where $\overrightarrow{Q} \colonequals \overrightarrow{oq}$ and $\overrightarrow{P} \colonequals \overrightarrow{op}$. Since both sets of the intersection are $(n - 1)$-spheres centred at $o$ and $p$ respectively, $\tri(o,p)$ is an $(n-2)$-sphere. See Figure~\ref{tri} for a diagram that illustrates this.
    
     \begin{figure}
         \centering
         \resizebox{0.7\textwidth}{!}{\input{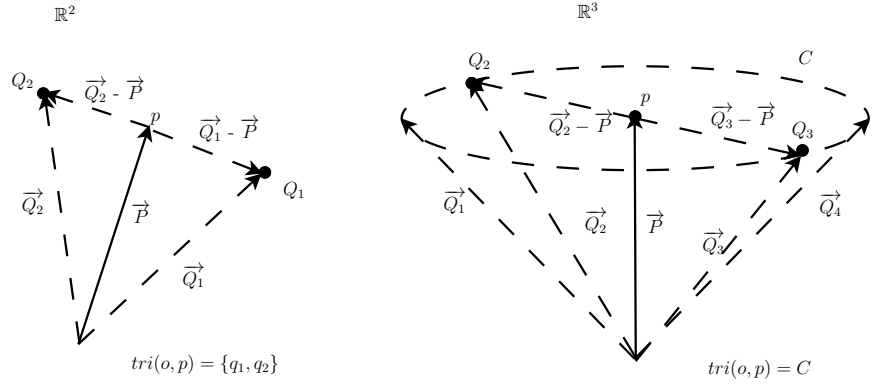}}
         \caption{The set $\tri(o,p)$ in $\mathbb{R}^{2}$ and $\mathbb{R}^{3}$ respectively. For the former space, this set is the 0-sphere $\{q_{1},q_{2}\}$. For the latter space, this set is a a 1-sphere.} 
         \label{tri}
     \end{figure}
    
    In order to prove Theorem~\ref{The1}, it suffices to show that there exists some $p \in J - \{o\}$ such that $J\cap\tri(o,p)\neq\varnothing$.
    
    Here is a more detailed outline of the path we take to prove our main result.
    
    \begin{itemize}
    \item \textbf{Setup:} In~\S\ref{I}, we introduce for $\tri(o,\gamma(t)) \colonequals S_{t}$, a scaled isometry $I_{t}$ such that $I_{t}(S_{t})$ is equal to the same $(n-2)$-sphere $S^{n-2}_{\can}$ for all $t \in (0,1)$. This isometry reorganises our coordinate system so that $S_{t}$ is mapped to a constant frame of reference for any $t \in (0,1)$. Since we consider the complement of $S_{t}$ in $\mathbb{R}^{n}$, it will be more convenient to reorganise our coordinate system to make $S_{t}$ appear 
    stationary, rather than considering a moving complement.
    In \S\ref{homlemmasection}, we assume that $\gamma$ and $S_{t}$ never intersect. Under this assumption, we prove that $I_{t}(\gamma) \simeq I_{t'}(\gamma)$ in $X \colonequals \mathbb{R}^{n} - S^{n-2}_{\can}$ for all $t,t' \in (0,1)$, where $\simeq$ denotes that two loops are freely homotopic in $X$.
    \item \textbf{Finding $t_{1}$:} In \S\ref{HomIdsection} we prove Lemma~\ref{HomId}, which states that $\gamma'_{t_{1}} \colonequals I_{t_{1}}(\gamma)$ is homotopic in $X$ to the trivial loop $c$ at $I_{t_{1}}(o)$ for some $t_{1} \in (0,1)$.
    \item \textbf{Finding $t_{2}$:} In \S\ref{NotHomIdsection}, we prove a series of technical lemmata (namely, Lemmata~\ref{Comp} to ~\ref{Flatnothom}) leading to a proof of Lemma~\ref{NotHomId}, which states that $\gamma'_{t_{2}}$ is not homotopic in $X$ to $c$ for some $t_{2} \in (0,1)$. We combine Lemmata~\ref{HomId} and~\ref{NotHomId} to show that $\gamma'_{t_{1}} \not\simeq \gamma'_{t_{2}}$ if $\gamma \cap S_{t} = \varnothing$ for every $t \in (0,1)$. However, from Lemma~\ref{cont}, we also know that $\gamma'_{t_{1}} \simeq \gamma'_{t_{2}}$ if $\gamma \cap S_{t} = \varnothing$ for every $t \in (0,1)$. Since both of these results follow from the same hypothesis---namely, that $\gamma \cap S_{t} = \varnothing$ for every $t \in (0,1)$---we know then that this hypothesis is false. We then prove Theorem~\ref{The1} and an additional corollary.
    \item \textbf{Inscribing equilateral triangles:} In \S\ref{last} we generalise Theorem~\ref{Mey} by showing that any point $o \in J$ inscribes an equilateral triangle if $J$ satisfies a certain condition.
    \end{itemize}
    
     See Figure~\ref{after} for diagrams outlining the sketch of the proof.
    \begin{figure}
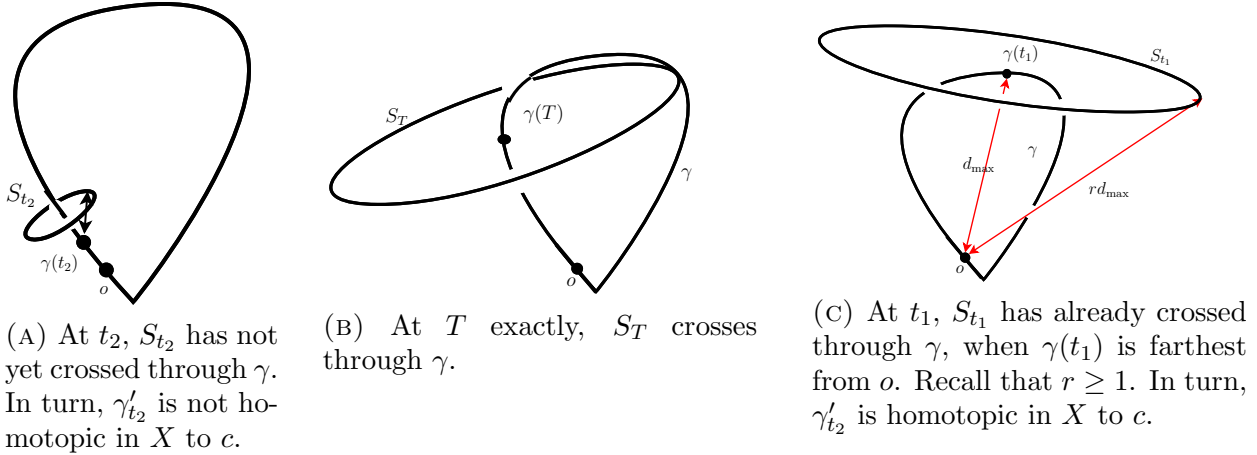

\begin{subfigure}{0.22\textwidth}
\resizebox{0.9\linewidth}{!}{\input{Diagram2prime.tex}} 
\caption{At $t_{2}$, $S_{t_{2}}$ has not yet crossed through $\gamma$. In turn, $\gamma'_{t_{2}}$ is not homotopic in $X$ to $c$.}
\label{before}
\end{subfigure}\hfill
\begin{subfigure}{0.35\textwidth}
\resizebox{0.9\linewidth}{!}{\input{Diagram1000prime.tex}} 
\caption{At $T$ exactly, $S_{T}$ crosses through $\gamma$.}\hfill
\end{subfigure}
\hfill
\begin{subfigure}{0.35\textwidth}
\resizebox{0.9\linewidth}{!}{\input{Diagram1000.tex}}
\caption{At $t_{1}$, $S_{t_{1}}$ has already crossed through $\gamma$, when $\gamma(t_{1})$ is farthest from $o$. Recall that $r \geq 1$. In turn, $\gamma'_{t_{2}}$ is homotopic in $X$ to $c$.}
\label{during}
\end{subfigure}
\caption{As $t$ increases from $t_{1}$ to $t_{2}$, and as $\gamma(t)$ goes farther away from $o$, $S_{T}$ intersects $\gamma$ for some $T \in (0,1)$.}
\label{after}
\end{figure}
     
 \section{Setup}\label{ConCon}
 \subsection{Translating $S_{t}$ to $S^{n-2}_{\can}$}\label{I}
 We show here that $S_{t}$ can always be mapped onto $S^{n-2}_{\can}$ by a scaled isometry $I_{t}$, where 
  $$S^{n-2}_{\can} = \set{(a_{1}, \dots, a_{n}) \in \mathbb{R}^{n}}{\sum_{i = 0}^{n-1} a_{i}^{2} = 1, a_{n} = 0}.$$
 The precise choice of $I_t$ is not important; we need only that $I_\bullet$ is a continuous map from $S^1$ to $\SO(n)\times\RR_{>0}$ such that $I_t(S_t)=S_{\can}^{n-2}$ for all $t$, but we provide a concrete description of one such family $I_\bullet$ nonetheless. We construct $I_{t}$ by composing a translation function $T_{t}:\mathbb{R}^{n} \rightarrow \mathbb{R}^{n}$, a rotation function $R_{t}:\mathbb{R}^{n} \rightarrow \mathbb{R}^{n}$ (whose center of rotation is $o$), and a scaling function $s_{t}:\mathbb{R}^{n} \rightarrow \mathbb{R}^{n}$. Here is how we construct each of the three functions.
 \begin{enumerate}
     \item $\mathbf{T_{t}}$: Let $o'$ be the centre of $S_{t}$. We set $T_{t}$ as the translation of $\mathbb{R}^{n}$ that maps $o'$ to $o$. 
     \item $\mathbf{R_{t}}$: Let $\Pi_{t}$ be the hyperplane that contains $S_{t}$, let $\mathbb{R}^{n}_{0} \colonequals \{(a_{1},a_{2},\dots,a_{n}) \in \mathbb{R}^{n} \mid a_{n} = 0\}$, and for any $(n - 1)$-plane $P \in \mathbb{R}^{n}$, let $N(P) = (\theta_{1}, \theta_{2}, \dots, \theta_{n-1}, 1)$
     denote the unit normal vector of that plane expressed by its $(n-1)$ spherical coordinates. Additionally, let us assume that $N(\mathbb{R}^{n}_{0}) = (0, 0, \dots, 0, 1)$, without a loss of generality, and let $N(\Pi_{t}) = (\phi_{1}, \phi_{2}, \dots, \phi_{n-1}, 1)$. We then set $R_{t}: \mathbb{R}^{n} \rightarrow \mathbb{R}^{n}$ to be the rotation of $\mathbb{R}^{n}$ about $o$ defined by the equation 
     $$R_{t}((\theta_{1}, \theta_{2}, \dots, \theta_{n-1}, r)) = (\theta_{1} - \phi_{1}, \theta_{2} - \phi_{2}, \dots, \theta_{n-1} - \phi_{n-1}, r).$$
     
     By this construction, $N \circ R_{t} \circ T_{t}(\Pi_{t}) = N(\mathbb{R}^{n}_{0})$. Thus, since every $(n - 1)$-plane passing through $o$ is uniquely identified by its unit normal vector, it follows that $R_{t} \circ T_{t}(S_{t}) = \mathbb{R}^{n}_{0}$ lies completely in $\mathbb{R}^{n}_{0}$, and is centred at $o$. Thus, $R_{t} \circ T_{t}(S_{t})$ is a nonzero scaling of $S^{n-2}_{\can}$.
     \item $\mathbf{s_{t}}$: We set $s_{t}$ as the scaling of $\mathbb{R}^{n}$ for which $s_{t} \circ (R_{t} \circ T_{t})(S_{t}) = I_{t}(S_{t}) = S^{n-2}_{\can}$.
 \end{enumerate}
 Figures~\ref{translation},~\ref{fig:rot}, and~\ref{fig:scaling} illustrate these transformations.
 
\begin{figure}
         \centering
         \resizebox{0.5\textwidth}{!}{\input{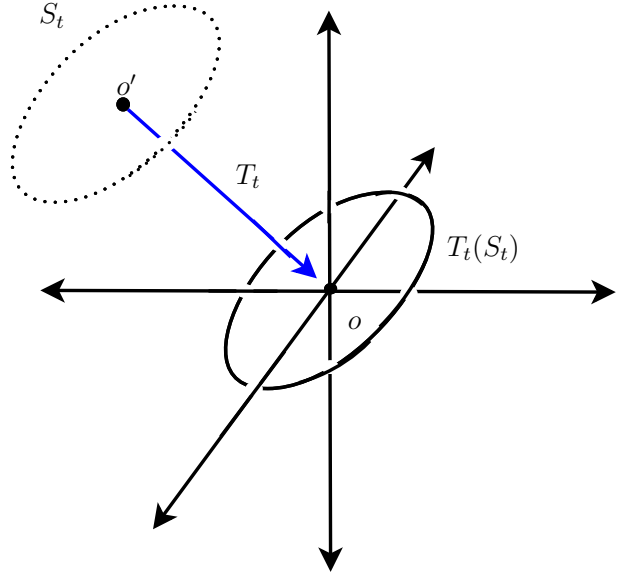}}
         \caption{The translation $T_{t}$ maps $o'$ to $o$.} \label{translation}
     \end{figure}

\begin{figure}
\begin{subfigure}{0.4\textwidth}
\resizebox{0.9\linewidth}{!}{\input{Diagram101.tex}} 
\caption{The rotation $R_{t}$ maps $\Pi_{t}$ to $\mathbb{R}^{n}_{0}$.}
\label{fig:rot}
\end{subfigure}\hfill
\begin{subfigure}{0.6\textwidth}
\resizebox{0.9\linewidth}{!}{\input{Diagram102.tex}} 
\caption{The scaling function $s_{t}$ maps $R_{t} \circ T_{t}(S_{t})$ to $S^{n-2}_{\can}$.}
\label{fig:scaling}
\end{subfigure}
\caption{}
\label{fig:transformations}
\end{figure}
     
  Note that whilst $I_{t}$ always maps $S_{t}$ to the same set $S^{n-2}_{\can}$, it also maps the loop $\gamma$ to a different loop $\gamma'_{t}:=I_t(\gamma)$ for each $t \in (0,1)$. 
 
 \subsection{A homotopy lemma}\label{homlemmasection}
 
 \begin{Lemma}\label{cont} 
 Suppose there is no $t \in (0,1)$ such that $\gamma\cap S_t\neq\varnothing$. Then for any $t,t'\in (0,1)$, we have $\gamma'_{t} \simeq \gamma'_{t'}$ in $X$.
 \end{Lemma}

 \begin{proof}
Let $F_{t,t'}(\cdot,T) \colonequals \gamma'_{(1-T)t + (T)t'}$. This is a homotopy taking $\gamma'_{t}$ (at $T = 0$) to $\gamma'_{t'}$ (at $T = 1$). For any $T \in [0,1]$, $F_{t,t'}(\cdot,T)$ is simply $\gamma'_{t''}$ for some $t'' \in (0,1)$. Thus, since no curve $\gamma'_{t''}$ intersects $X$, $F_{t,t'}$ is a homotopy in $X$.
\end{proof}

\section{Finding $t_{1}$} \label{HomIdsection}

\begin{Notation} Throughout the rest of the paper, we let $c$ denote the constant path at $o$. \end{Notation}

Let us choose $t_1$ such that $\|\gamma(t_1)-o\|:=d_{\max}$ is maximal. We shall prove the following lemma regarding $t_1$:

\begin{Lemma}\label{HomId}
Assume there is no $t \in (0,1)$ such that $S_{t}$ intersects $J$. Then $\gamma'_{t_{1}}$ is homotopic in $X$ to $c$.
\end{Lemma}

The idea is as follows. Let $x$ denote an arbitrary point in $S_{t_{1}}$. Then we have $\| x - o \| = rd_{\max}$, since $r \geq 1$. Then $\| \gamma(t) - o \| \leq \| x - o \|$ for any $t \in (0,1)$. Thus, $S_{t_{1}}$ has already slipped out from $\gamma$, being far enough from $\gamma$ that $\gamma$ can be shrunk to $c$ by a linear homotopy $L$ that does not intersect $S_{t}$. See Figure~\ref{fig:image2}.

\begin{figure}
\begin{subfigure}{0.45\textwidth}
\resizebox{0.9\linewidth}{!}{\input{Diagram3prime.tex}} 
\caption{As $r \geq 1$, $rd_{\max} \geq d_{\max}$. Thus, $S_{t_{1}}$ is so far that it cannot intersect $\gamma$ as it is homotoped to $c$ at $o$.}
\label{homtoid}
\end{subfigure}\hfill
\begin{subfigure}{0.45\textwidth}
\resizebox{0.9\linewidth}{!}{\input{Diagram1001.tex}} 
\caption{$\gamma$ being homotoped to $c$ by a straight line homotopy.}
\label{fig:subim1}
\end{subfigure}
\caption{}
\label{fig:image2}
\end{figure}

\begin{proof}[Proof of Lemma~\ref{HomId}]
We first show that the straight-line homotopy $L:[0,1] \times [0,1] \rightarrow \mathbb{R}^{n}$ from $\gamma$ to $c$ never intersects $S_{t_{1}}$. We consider two cases for $L(s, \cdot)$; first, when $t = 0$, and second, when $t \in (0,1]$.
\begin{itemize}
    \item $t = 0$: $L(s,0) = \gamma$ does not intersect $S_{t_{1}}$ for any $s \in [0,1]$ because of our assumption that no such intersection occurs. 
    \item  $t \in (0,1]$: Since $L$ is a straight-line homotopy to $o$, $\| L(s,\cdot) - o \|$ must be monotonically decreasing over $(0,1]$, so in particular
  $$\| L(s,t) - o \| < d_{\max} \leq \| x - o \|.$$
 That is, the distance of $L(s,t)$ from $o$ is always less than that of $x$ from $o$. Thus, $L$ never intersects $S_{t_{1}}$ when $t \in (0,1]$.
 \end{itemize}
 We conclude that $L(s,t)$ does not intersect $S_{t_{1}}$ for any $s,t \in [0,1]$.
 
 Thus $I_{t_{1}}$ induces the desired homotopy $I_{t_{1}} \circ L$ taking $I_{t_{1}} \circ \gamma = \gamma'_{t_{1}}$ to $I_{t_{1}} \circ c \colonequals c'$. Since $L$ never intersects $S_{t_{1}}$, and since $I_{t_{1}}$ is a bijection, $I_{t_{1}} \circ L$ never intersects $I_{t_{1}}(S_{t_{1}}) = S^{n-2}_{\can}$. Therefore, $\gamma'_{t_{1}} \simeq c'$ in $X$ via the homotopy $I_{t_{1}} \circ L$. 
 \end{proof}
 
 \section{Finding $t_{2}$} \label{NotHomIdsection}
 
\begin{Lemma}\label{NotHomId}
 Let $\theta_{v}$ be the angle of the vertex of $v$ the triangle $\triangle$ to be inscribed. Suppose there is no $t \in (0,1)$ such that $\gamma$ intersects $S_{t}$. If $\theta_{v}$ is such that
 $$\limsup_{\delta\rightarrow0^{+}} \Theta_{\delta} < \theta_{v} <\liminf_{\delta\rightarrow0^{+}} \Theta'_{\delta},$$
 then there exists some $t_{2} \in (0,1)$ has the property that $\gamma'_{t_{2}} \not\simeq c'$ in $X$.
 \end{Lemma}

The proof involves proving a series of lemmata (namely, Lemmata~\ref{Comp} to ~\ref{Flatnothom}) leading to a proof of Lemma~\ref{NotHomId}. We begin by recalling a well-known and simple preliminary result.

\begin{Lemma}\label{Comp}
     Let $A,B \subset \mathbb{R}^{n}$ be two disjoint compact sets. There is some $\inf(A,B)>0$ such that $\| a - b \| \geq \inf(A,B)$ for any $a \in A$ and any $b \in B$.
 \end{Lemma}

For each of the remaining lemmata (from Lemma~\ref{first} to Lemma~\ref{Flatnothom}), we assume the same hypotheses as are assumed for Lemma~\ref{NotHomId}, namely that there is no $t \in (0,1)$ such that $\gamma$ intersects $S_{t}$, and that the angle $\theta_{v}$ of some vertex $v$ of $\triangle$ is such that $\limsup_{\delta\rightarrow0^{+}} \Theta_{\delta} < \theta_{v} <\liminf_{\delta\rightarrow0^{+}} \Theta'_{\delta}$.

\begin{Lemma}\label{first}
There exists some $\epsilon>0$ such that $\sup\Theta_{\varepsilon} < \theta_{v} < \inf\Theta'_{\varepsilon}$.
\end{Lemma}

\begin{proof} 
Since $\limsup_{\delta\rightarrow0^{+}}\Theta_{\delta} < \theta_{v} < \liminf_{\delta\rightarrow0^{+}}\Theta'_{\delta}$, there exist arbitrarily small values of $\varepsilon_{1},\varepsilon_{2} > 0$ such that $\sup\Theta_{\varepsilon_{1}} < \theta_{v} < \inf\Theta'_{\varepsilon_{2}}$. Let $\varepsilon \colonequals \min(\varepsilon_{1},\varepsilon_{2})$. Then, $\sup\Theta_{\varepsilon} < \theta_{v} < \inf\Theta'_{\varepsilon}$. 
\end{proof}

\begin{Lemma}\label{nhoodinsphere}
 For a positive real number $r$, let $B(o,r)$ be the closed $n$-ball of radius $r$, centred at $o$. Additionally, let the neighbourhood $N \colonequals \Ima \gamma\mid_{(1 - \epsilon, \epsilon)}$. Then there exists $d_{\min}>0$ such that $J \cap B(o,d_{\min}) \subset N$.
\end{Lemma}

Note that $(1 - \epsilon, \epsilon)$ is taken modulo 1, i.e.\ $(1-\epsilon,\epsilon)=[0,\epsilon)\cup(1-\epsilon,1]$.

\begin{proof}
Let $d_{\min} = \frac{\inf(\{o\},J - N)}{2}$. Since $\{o\}$ and $J - N$ are disjoint and compact, Lemma~\ref{Comp} tells us that $d_{\min} > 0$. Since $d_{\min} < \inf(\{o\},J - N)$, it thus follows that $(J - N) \cap B(o,d_{\min})=\varnothing$, and thus that $J \cap B(o,d_{\min}) \subset N$.
\end{proof}

\begin{Lemma}\label{sphereboundary}
 There exists some $t_{2} \in [0,1]$ such that $S_{t_{2}} \subset \partial B(o,d_{\min})$.
\end{Lemma}
\begin{proof}
Let $p \in N$ be a point such that $\| p - o \| = d \colonequals \frac{d_{\min}}{r}$. We let $t_{2}\in(0,1)$ be such that $p = \gamma(t_{2})$. Then all the points of $S_{t_{2}}$ lie at a distance $rd = d_{\min}$ from $o$. Therefore $S_{t_{2}} \subset \partial B(o,d_{\min})$.
\end{proof}

\begin{Definition}
We say a point lies in \emph{the interior (or exterior) of $S_{t_{2}}$} when it lies within \emph{the interior (or exterior) of $S_{t_{2}}$ when $\Pi_{t_{2}}$ is taken as its ambient space}. 
\end{Definition}

\begin{Lemma}\label{interiorinsphere}
 The interior of $S_{t_{2}}$ lies in $B(o,d_{\min})$.
\end{Lemma}
\begin{proof}
 Since (from Lemma~\ref{sphereboundary}) $S_{t_{2}} \subset \partial B(o,d_{\min})$, it follows that its interior is a subset of $B(o,d_{\min})$ too.
\end{proof}

\begin{Lemma}\label{out}
All intersections of $J - N$ with $\Pi_{t_{2}}$ lie in the exterior of $S_{t_{2}}$.
\end{Lemma}
\begin{proof}
 Assume that some $p \in (J - N) \cap \Pi_{t_{2}}$ lies in the interior of $S_{t_{2}}$. Since (from Lemma~\ref{interiorinsphere}) the interior of $S_{t_{2}}$ is a subset of $B(o,d_{\min})$, it follows that $p \in (J - N) \cap B(o,d_{\min})$. However, from Lemma~\ref{nhoodinsphere}, $J \cap B(o, d_{\min}) \subset N$, contradicting the assumption that $p$ lies in the interior of $S_{t_2}$.
\end{proof}

See Figure~\ref{ring} for an illustration of the preceding Lemmata.

\begin{figure}
    \centering
    \resizebox{0.65\linewidth}{!}{\input{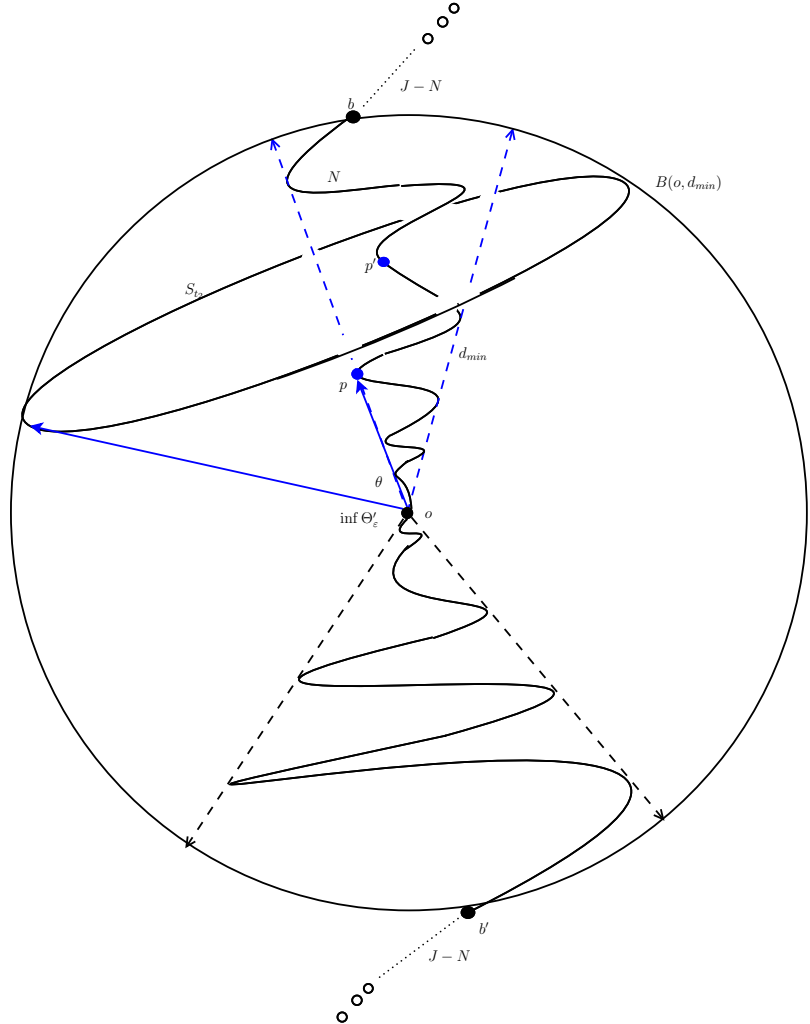}}
    \caption{Here, the angle between the blue lines is $\sup \Theta_{\varepsilon}$. The angle between the left-facing blue line and the left-facing black line is $\inf \Theta'_{\varepsilon}$. Since $\sup\Theta_{\varepsilon} < \theta_{v} < \inf\Theta'_{\varepsilon}$, it follows $S_{t_{2}}$ is wide enough such that $N$ always intersects $\Pi_{t_{2}}$ at the interior of $S_{t_{2}}$. Here, $p'$ denotes one such point of intersection.}
    \label{ring}
\end{figure}

 Since $I_{t_{2}}$ is an isometry, all of the Lemmata above (which apply for $\gamma$) also apply to the curve $\gamma'_{t_{2}} = I_{t_{2}} \circ \gamma$---for example, $\gamma'_{t_{2}}\mid_{[1 - \epsilon, \epsilon]}$ intersects $I_{t_{2}}(\Pi_{t_{2}}) = \mathbb{R}^{n}_{0}$ only at the exterior of $I_{t_{2}}(S_{t_{2}}) = S^{n-2}_{\can}$.

\begin{Definition}
     Let $P: \mathbb{R}^n \rightarrow \mathbb{R}_{\geq 0} \times \mathbb{R}$ be defined by
    $$P((r_{1}, r_{2}, \dots, r_{n})) = (d(r), r_{n}),$$
    where $d: \mathbb{R}^{n-1} \rightarrow \mathbb{R}_{\geq 0}$ is defined to be
    $$d(r)=\sqrt{r_1^2+r_2^2+\cdots+r_{n-1}^2}.$$
\end{Definition}

\begin{Definition}
      For any given $p \in \mathbb{R}^{2}\setminus\Ima f$, let a path $f: [0,1]$ be parametrized to polar form $(r(t), \theta(t))$, where $r(t) = \|f(t) - p\|$ and $\theta(t)$ continuously maps $t$ to the angle of the segment $pf(t)$ relative to the positive vertical axis from $p$. Then, we define the winding number $\eta_{p}(f)$ of $f$ relative to $p$ by the equation
      $$\frac{\theta(1) - \theta(0)}{2\pi}.$$
\end{Definition}

\begin{Lemma}\label{antic1}
 $\eta_{(1,0)}(P' \circ \gamma'_{t_{2}}\mid_{[1 - \epsilon, \epsilon]}) < 0$.
 \end{Lemma}
 
 \begin{proof}
To begin with, note that $P(S^{n-2}_{\can}) = (1,0)$. Let $(0,1) = u$, and let $\phi(t)$ be the angle that the segment $uP \circ \gamma'_{t_{2}}(t)$ makes relative to the segment $uP \circ \gamma'_{t_{2}}(0)$. See Figure~\ref{compressedpath}.

Assume, for the sake of contradiction, that $\eta_{(1,0)}(P \circ \gamma'_{t_{2}}\mid_{[1 - \epsilon, 0]}) > 0$
Then, since $\sup \Theta_{\epsilon} < \theta_{v}$, it follows $\phi(\epsilon) < \theta_{v}$. Also, since $\inf \Theta'_{\epsilon} > \theta_{v}$, it follows that $\phi(1 - \epsilon) > \theta_{v}$. 
 
 \begin{figure}
    \centering
    \resizebox{0.3\linewidth}{!}{\input{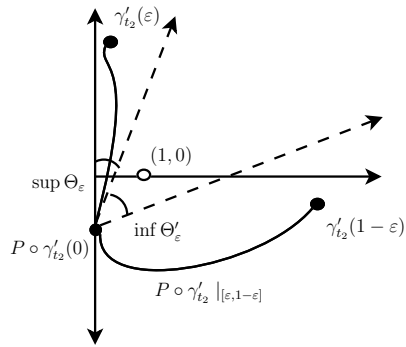}}
    \caption{The path $P \circ \gamma'_{t_{2}}\mid_{[1 - \epsilon, \epsilon]}$. This path cannot have a negative winding number, since that would contradict the fact that $\phi(t) \geq \inf \Theta'_{\epsilon}$ for all $t \in [1 - \epsilon, 0]$ and $\phi(t) \leq \sup \Theta_{\epsilon}$ for all $t \in [0, \epsilon]$.}
    \label{compressedpath}
\end{figure}

 As $\phi(1 - \epsilon) > \phi(\epsilon)$, if $\eta_{(1,0)}(P \circ \gamma'_{t_{2}}\mid_{[1 - \epsilon, 1 - \epsilon]}) > 0$ (i.e.\ if the net movement is anticlockwise around $p$) then $\phi(t) = \theta_{v}$ for some $t \in [\epsilon, 1 - \epsilon]$, in contradiction to the fact that $\phi(t) \geq \inf \Theta'_{\epsilon}$ for all $t \in [1 - \epsilon, 0]$ $\phi(t) \leq \sup \Theta_{\epsilon}$ for all $t \in [0, \epsilon]$.
 \end{proof}
 
 \begin{Definition}
Let $a_{n}: \mathbb{R}^n \rightarrow \mathbb{R}$ be the last-coordinate map, i.e.\ $a_{n}((r_{1}, r_{2}, \ldots, r_{n})) = r_{n}$. 
\end{Definition}
 
 \begin{Lemma}\label{above}
 $a_{2} \circ P \circ \gamma'_{t_{2}}(\epsilon) > 0$.
\end{Lemma}

\begin{proof}
 Assume that this is not true. Then, it would follow that $\|o-\gamma'_{t_{2}}(\epsilon)\| < \|o-u\|$, which implies that $\gamma'_{t_{2}}(\epsilon)$ lies in the interior of $B(o,d_{\min})$. This would, in turn, also imply that $J _ N$ (in whose closure $\gamma'_{t_{2}}(\epsilon)$ belongs) has a nonempty intersection with $B(o,d_{\min})$. However, this contradicts the definition of $N$ in Lemma~\ref{nhoodinsphere}.
\end{proof}
 
 \begin{Lemma}\label{antic2}
  $\eta_{(1,0)}(P \circ \gamma'_{t_{2}}\mid_{[\epsilon, 1 - \epsilon]}) < 0$.
 \end{Lemma}
 
 \begin{proof}
    Let $p_{0} = (a,0)$, where $0 \leq a < 1$. Then, the angle that $up_{0}$ makes against the positive vertical axis from $u$ is $\frac{\pi}{2}$. From Lemma~\ref{above}, we infer that $\phi(\epsilon) < \frac{\pi}{2}$, and thus that $\phi(\epsilon) - \frac{\pi}{2} < 0$. See Figure~\ref{compressedpath1}. 
    
       \begin{figure}
    \centering
    \resizebox{0.3\linewidth}{!}{\input{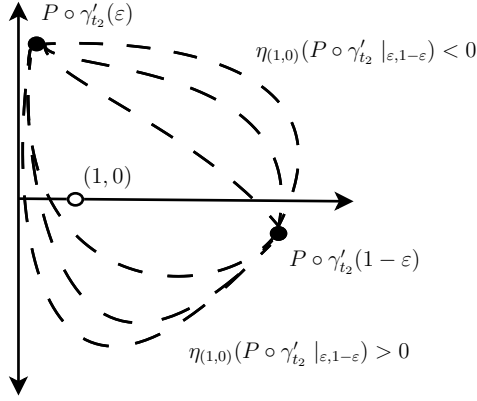}}
    \caption{$\eta_{(1,0)}(P \circ \gamma'_{t_{2}}\mid_{\epsilon, 1 - \epsilon}) > 0$ if and only if the curve passes $(0,1)$ by the left, whilst $\eta_{(1,0)}(P \circ \gamma'_{t_{2}}\mid_{\epsilon, 1 - \epsilon}) < 0$ if and only if it passes it from the right.}
    \label{compressedpath1}
\end{figure}
    
    Since we also know that $\frac{\pi}{2} - \phi(1 - \epsilon) < 0$, it follows that if $\eta_{(1,0)}(P \circ \gamma'_{t_{2}}\mid_{[1 - \epsilon, \epsilon]}) > 0$, then there is some $t' \in [\epsilon, 1 - \epsilon]$ such that $\phi(t') = \frac{\pi}{2}$, i.e. that $P \circ \gamma'_{t_{2}}(t') = (a,0)$ for some $0 \leq a < 1$. However, it would then follow that $\gamma'_{t_{2}}(t') \in \gamma'_{t_{2}}\mid_{[\epsilon, 1 - \epsilon]} = J - N$ lies within the interior of $S^{n-2}_{can}$, which is in contradiction to Lemma~\ref{out}. 
 \end{proof}
 
 \begin{Lemma}\label{Flatnothom}
 $P \circ \gamma'_{t_{2}} \not\simeq P \circ c' \colonequals c''$ in $P(X) = \mathbb{R}^{2} - (1,0)$.
 \end{Lemma}
 
 \begin{proof}
Since $\eta_{(1,0)}(P \circ \gamma'_{t_{2}}) = \eta_{(1,0)}(P \circ \gamma'_{t_{2}}\mid_{[1 - \epsilon, \epsilon]}) + \eta_{(1,0)}(P \circ \gamma'_{t_{2}}\mid_{[\epsilon, 1 - \epsilon]})$, Lemmata~\ref{antic1} and~\ref{antic2} imply that $\eta_{(1,0)}(P \circ \gamma'_{t_{2}}) \neq 0$, which is true if and only if this curve is not nullhomotopic.
\end{proof}

\begin{proof}[Proof of Lemma~\ref{NotHomId}]
Assume that $\gamma'_{t_{2}} \simeq c'$. Let $H: [0,1] \times [0,1] \rightarrow \mathbb{R}_{\geq 0} \times \mathbb{R}$ be a homotopy such that $H(0, \cdot) = \gamma'_{t_{2}}$ and $H(1, \cdot)  = c'$. Then, the homotopy $P \circ H$ takes $P \circ H(0, \cdot) = P \circ \gamma'_{t_{2}}$ to the constant path $P \circ H(1, \cdot) = c''$. However, from Lemma~\ref{Flatnothom}, we know that this is impossible.
\end{proof}

Therefore, Lemmata~\ref{HomId} and~\ref{NotHomId} immediately imply Theorem~\ref{The1}.

As a corollary of independent interest of this theorem, we get the following.

\begin{Theorem}\label{free}
If $o \in J$ is a differentiable point, then any triangle $\triangle$ can be inscribed in $J$.
\end{Theorem}
\begin{proof}
When $o$ is differentiable, $\limsup_{\delta\rightarrow0^{+}}\Theta_{\delta} = 0$ and $\liminf_{\delta\rightarrow0^{+}}\Theta'_{\delta} = \pi$. Then, the angle $\theta_{v}$ of any vertex of $\triangle$ always satisfies $\limsup_{\delta\rightarrow0^{+}}\Theta_{\delta} < \theta_{v} < \liminf_{\delta\rightarrow0^{+}}\Theta'_{\delta}$. Thus, $\triangle$ always satisfies the requirements of Theorem~\ref{The1} to be inscribed in $J$.
\end{proof}

\section{Generalising theorem~\ref{Mey}}\label{last}
We now prove that subject to certain restrictions, Theorem~\ref{Mey} can be generalised to any given Jordan curve $J \subset \mathbb{R}^{n}$. We state this condition before stating our generalisation.
\begin{Definition}
     For any given $o \in J$, and any $p \in J - o$, let $a_{p}: [0,1] \rightarrow \mathbb{R}$ be defined as 
     $$a_{p}(t) = \frac{\vv{o\gamma(t)}}{\|o\gamma(t)\|} \cdot \vv{op}.$$ 
     If there is some $\epsilon > 0$ such that $a_{p}\mid_{(1 - \epsilon, \epsilon)}$ is monotone for each $p \in \Ima \gamma\mid_{(1 - \epsilon, \epsilon)} - o$, we say that $J$ is \emph{strongly monotone} at $o$. Here, $\gamma\mid_{(1 - \epsilon, \epsilon)}$ is called a \emph{strongly locally monotone neighbourhood} of $o$.
\end{Definition}

Examples wherein similar conditions have been used to prove inscription theorems can be found in, for example,~\cite{stromquist_1989}. See Figure~\ref{sharpedge} for examples of neighbourhoods that both satisfy and fall outside this condition.

\begin{figure}
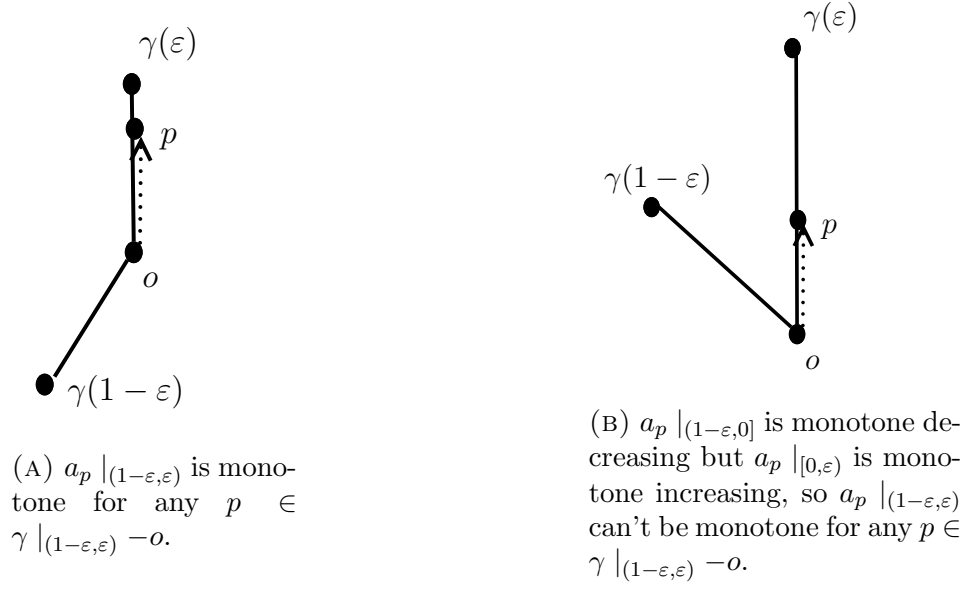

\hfill
\begin{subfigure}{0.23\textwidth}
\resizebox{0.9\linewidth}{!}{\input{notsharp.tex}} 
\caption{$a_{p}\mid_{(1 - \epsilon, \epsilon)}$ is monotone for any $p \in \gamma\mid_{(1 - \epsilon, \epsilon)} - o$.}
\label{sharpdia}
\end{subfigure}
\hfill
    \begin{subfigure}{0.3\textwidth}
\resizebox{0.9\linewidth}{!}{\input{sharp.tex}} 
\caption{$a_{p}\mid_{(1 - \epsilon,0]}$ is monotone decreasing but $a_{p}\mid_{[0,\epsilon)}$ is monotone increasing, so $a_{p}\mid_{(1 - \epsilon, \epsilon)}$ can't be monotone for any $p \in \gamma\mid_{(1 - \epsilon, \epsilon)} - o$.}
\label{notsharpdia}
\end{subfigure}
\hfill
\caption{$J$ is strongly locally monotone at $o$ in Figure~\ref{sharpdia}, and is not in Figure~\ref{notsharpdia}.}
\label{sharpedge}
\end{figure}

\begin{Theorem}\label{General}
If $J$ is strongly locally monotone at $o$, then the point $o \in J$ has an inscribed inscribed equilateral triangle. 
\end{Theorem}

Note that, in comparison to Theorem~\ref{The1} for an equilateral triangle, this theorem is stronger, since it allows for an unaccountably infinite number of equilateral triangles to be inscribed at the strongly monotone neighbourhood containing the point $o$, since any point of that neighbourhood can be chosen as the first point for the equilateral triangle. By contrast, finding a suitable neighbourhood that satisfies the hypothesis of Theorem~\ref{The1} only ensures at least a single triangle of the desired specifications can be inscribed in it. 

For the proof, our first concern is how we identify when $o$ inscribes an equilateral triangle. Let $\phi_{s}:[0,1] \rightarrow [0,s]$ be defined by $\phi_{s}(t) = st$. For all $s \in (0,1)$, let $r_{1,s}: [0,1] \rightarrow \mathbb{R}_{\geq 0}$ and $r_{2,s}: [0,1] \rightarrow \mathbb{R}_{\geq 0}$ be respectively defined by 
\[ r_{1,s}(t) = \frac{\| \gamma \circ \phi_{s}(t) - o \|}{\| \gamma(s) - o \|} \text{   and   } r_{2,s}(t) = \frac{\| \gamma \circ \phi_{s}(t) - \gamma(s) \|}{\| \gamma(s) - o \|},\] where, for any $t \in [0,1]$, $r_{1,s}(t)$ and $r_{2,s}(t)$ respectively represent the ratio of the lengths of the sides $o\gamma \circ \phi_{s}(t)$ and $\gamma(s)\gamma \circ \phi_{s}(t)$ to the side $o\gamma(s)$. Then, the \emph{ratio path} $R_{s}: [0,1] \rightarrow \mathbb{R}^{2}$ for any $s \in (0,1)$ is
$$R_{s}(t) = (r_{1,s}(t) - 1, r_{2,s}(t) - 1).$$
To prove Theorem~\ref{General}, then, it suffices to show that, for some $s',t' \in (0,1)$, $r_{1,s'}(t') = r_{2,s'}(t') = 1$, and thus that $R_{s'}(t') = (0,0)$. Note that $R_{s}(0) = (-1,0)$ and $R_{s}(1) = (0,-1)$.

\begin{Definition}
For any pair of paths $f,g$, let $f\ast g$ denote the concatenation of the paths $f,g$.
\end{Definition}

\begin{Definition}
For any path $f$, let $\bar{f}$ denote the inverse of $f$, defined by $\bar{f}(t) = f(1 - t)$.
\end{Definition}

\begin{Definition}
For any $s,s' \in (0,1)$, let the loop $L_{s,s'}$ be defined by the equation $L_{s,s'} \colonequals R_{s} \ast \bar{R}_{s'}$.
\end{Definition}
 The approach is as follows. Let $c''$ denote the constant path at $R_{s}(0) = (-1,0)$. To begin, we use Lemma~\ref{prefalse} to prove Lemma~\ref{false}, which states that the loop $L_{s,s'} \simeq c''$ in $\mathbb{R}^{2} - (0,0)$ for any $s,s' \in (0,1)$ if $R_{s}$ doesn't contain $(0,1)$ for any $s \in (0,1)$. Then, we combine two technical lemmata---namely, Lemmata~\ref{negativex} and~\ref{positivequads}---to prove Lemma~\ref{true}, which states that $L_{s_{1},s_{2}} \not\simeq c''$ if $R_{s}$ doesn't contain $(0,0)$ for any $s \in (0,1)$. However, from Lemma~\ref{false}, we also know that $L_{s_{1},s_{2}} \simeq c''$ if $R_{s}$ doesn't contain $(0,0)$ for any $s \in (0,1)$. Since both of these results follow from the same hypothesis---namely, that $R_{s}$ doesn't contain $(0,0)$ for any $s \in (0,1)$---we know then that this hypothesis is false. A proof of Theorem~\ref{General} then immediately follows.
\begin{Lemma}\label{prefalse}
Assume $R_{s}(t)$ does not contain $(0,0)$ for any $s \in (0,1)$. Then, $R_{s} \simeq R_{s'}$ in $\mathbb{R}^{2} - (0,0) \colonequals X'$ for any $s,s' \in (0,1)$.
\end{Lemma}

\begin{proof}
Let $F'_{s,s'}(\cdot, T) \colonequals R_{s(1 - T) + s'T}$. This is a homotopy taking $R_{s}$ (at $T = 0$) to $R_{s'}$ ($T = 1$). Then, for each $T \in [0,1]$, $F'_{s,s'}(\cdot, T)$ is simply $R_{s''}$ for some $s'' \in [s,s']$. Thus, if no curve $R_{s''}$ intersects $(0,0)$, then $F_{s,s'}$ is the desired homotopy taking $R_{s}$ to $R_{s'}$ in $X'$.
\end{proof}

\begin{Lemma}\label{false}
Assume $R_{s}$ not contain $(0,0)$ for any $s \in (0,1)$. Then $L_{s,s'} \simeq c''$ in $X'$ for any $s,s' \in (0,1)$.
\end{Lemma}
\begin{proof}
From Lemma~\ref{prefalse}, $R_{s} \simeq R_{s'}$ if $R_{s}$ does not contain $(0,0)$ for any $s \in (0,1)$. Then, $L_{s,s'} = R_{s} * \bar{R}_{s'} \simeq R_{s'} * \bar{R}_{s'} \simeq c''$. 
\end{proof}
\begin{Lemma}\label{negativex}
There exists some $s_{1} \in (0,1)$ such that $R_{s_{1}} \subset \mathbb{R}_{\leq 0} \times \mathbb{R}$.
\end{Lemma}
\begin{proof}
Let $s_{1} \in (0,1)$ be such that $\| \gamma(s_{1}) - o \|$ is maximal. Then, by definition, $\| \gamma(t) - o \| \leq \| \gamma(s_{1}) - o \|$ for all $t \in [0,1]$, and thus $r_{1,s_{1}}(t) = \frac{\| \gamma \circ \phi_{s_{1}}(t) - o \|}{\| \gamma(s_{1}) - o \|} \leq 1$. Let $a_{1}: \mathbb{R}^{2} \rightarrow \mathbb{R}$ map each point to its first component. Then, $a_{1} \circ R_{s_{1}}(t) = r_{1,s_{1}}(t) - 1 \leq 0$, and thus $R_{s_{1}} \subset \mathbb{R}_{\leq 0} \times \mathbb{R}$.
\end{proof}

\begin{Lemma}\label{positivequads}
Let $J$ be strongly locally monotone at $o$. There exists some $s_{2} \in (0,1)$ such that $R_{s_{2}} \cap (\mathbb{R}_{<0} \times \mathbb{R}_{<0}) = \varnothing$.
\end{Lemma}

\begin{proof}
Let $\epsilon>0$ be such that $\gamma\mid_{(1 - \epsilon, \epsilon)} \colonequals U \subset J$ is a strongly locally monotone neighbourhood of $o$. Let $d \colonequals \frac{\inf(\{o\}, J - U)}{3}$, and let $s_{2} \in (0,\epsilon)$ be such that $\| \gamma(s_{2}) - o \| = d$. Note that $J - U \subset \gamma\mid_{[0,s_{2}]}$. Also, let $B'(o,r)$ denote the open $n$-ball of radius $r$ centred at $o$.

Then, $R_{s_{2}}(t) \in \mathbb{R}_{<0} \times \mathbb{R}_{<0}$ if and only if $\| \gamma(t) - o \| < d$ and $\| \gamma(t) - \gamma(s_{2}) \| < d$ for all $t \in [0,1]$, and thus, if and only if $\gamma \circ \phi_{s_{2}} \cap B \neq \varnothing$, where  $B \colonequals B'(o,d) \cap B'(\gamma(s_{2}),d)$. As such, it shall suffice to show that $\gamma \circ \phi_{s_{2}} \cap B = \varnothing$ to complete the proof. We consider three parts of $\gamma \circ \phi_{s_{2}}$ separately:$\gamma \circ \phi_{s_{2}}\mid_{[0, \frac{\epsilon}{s_{2}}]}$, $\gamma \circ \phi_{s_{2}}\mid_{[\frac{\epsilon}{s_{2}}, \frac{1 - \epsilon}{s_{2}}]} = J - U$, and $\gamma \circ \phi_{s_{2}}\mid_{[\frac{1 - \epsilon}{s_{2}}, 1]}$. 
\begin{itemize}
\item $\gamma \circ \phi_{s_{2}}\mid_{[0, \frac{\epsilon}{s_{2}}]}:$ Since (without a loss of generality) $a_{\gamma(s_{2})}$ is monotonically increasing, the starting point $o$ of $\gamma \circ \phi_{s_{2}}\mid_{[0, \frac{\epsilon}{s_{2}}]}$ is its lowest point. Thus, since $a_{\gamma(s_{2})}(o) > a_{\gamma(s_{2})}(b)$ for any $b \in B$, no point of $\gamma \circ \phi_{s_{2}}\mid_{[0, \frac{\epsilon}{s_{2}}]}$ lies close enough to the axis defined by $\vv{o\gamma(s_{2})}$ to intersect $B$. See Figure~\ref{monotone} for a diagram that illustrates this.
\begin{figure}
         \centering
         \resizebox{0.6\textwidth}{!}{\input{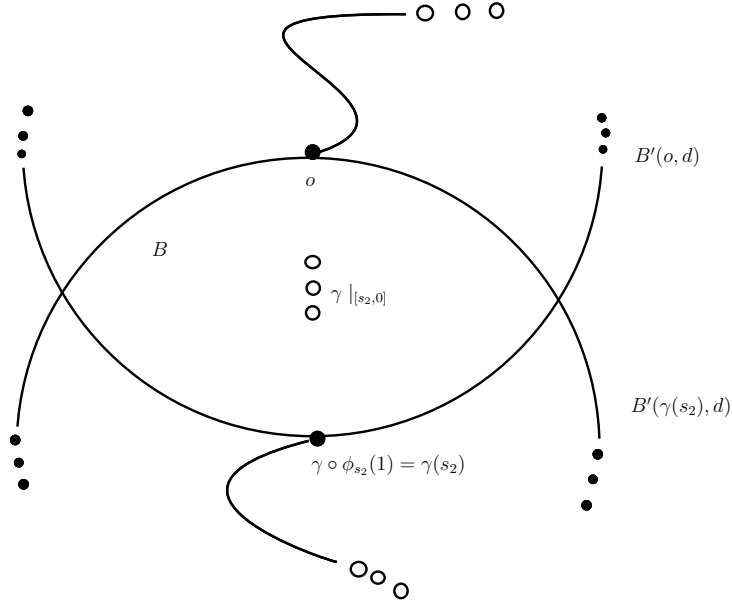}}
         \caption{No point of $\gamma \circ \phi_{s_{2}}\mid_{[0, \frac{\epsilon}{s_{2}}]}$ lies low enough to intersect $B$, and no point of $\gamma \circ \phi_{s_{2}}\mid_{[\frac{1 - \epsilon}{s_{2}}, 1]}$ is high enough to intersect $B$.}
         \label{monotone}
     \end{figure}
     
\item $\gamma \circ \phi_{s_{2}}\mid_{[\frac{\epsilon}{s_{2}}, \frac{1 - \epsilon}{s_{2}}]} = J - U$: Since  $d < \inf(\{o\},J - U)$, it follows that $(J - U) \cap B'(o,d) = \varnothing$. Then
$$\inf(\{\gamma(s_{2})\}, J- U) \geq \| \inf(\{o\}, J - U) \| - \| \gamma(s_{2}) - o \| = 3d - d = 2d.$$ Since $d < 2d \leq \inf(\gamma(s_{2}), J - U)$, it follows that $(J - U) \cap B'(\gamma(s_{2}),d) = \varnothing$. Therefore, $J - U$ intersects neither $B'(o,d)$ nor $B'(\gamma(s_{2}),d)$, and in particular, $(J - U) \cap B = \varnothing$. 

\item $\gamma \circ \phi_{s_{2}}\mid_{[\frac{1 - \epsilon}{s_{2}}, 1]}$: Since (without a loss of generality) $a_{\gamma(s_{2})}$ is monotonically increasing, the end point $\gamma \circ \phi_{s_{2}}(1) = \gamma(s_{2})$ of $\gamma \circ \phi_{s_{2}}\mid_{[\frac{1 - \epsilon}{s_{2}}, 1]}$ is its ``highest'' point. That is to say, $a_{s_{2}} \circ \gamma(s_{2}) > a_{\gamma(s_{2})} \circ \gamma \circ \phi_{s_{2}}\mid_{[\frac{1 - \epsilon}{s_{2}}, 1]}(t)$ for any $t$. Thus, since $a_{\gamma(s_{2})}(o) < a_{\gamma(s_{2})}(b)$ for any $b \in B$ and any $t$, $a_{\gamma_{s_{2}}} \circ \gamma \circ \phi_{s_{2}}\mid_{[0, \frac{\epsilon}{s_{2}}]}(t) < a_{\gamma(s_{2})}(b)$. In other words, no point of $\gamma \circ \phi_{s_{2}}\mid_{[0, \frac{\epsilon}{s_{2}}]}$ lies high enough (relative to the axis defined by $\vv{o\gamma(s_{2})}$) to intersect $B$. Again, see Figure~\ref{monotone}.
\end{itemize}

Therefore, $\gamma \circ \phi_{s_{2}} \cap B = \varnothing$, and thus $R_{s_{2}} \cap (\mathbb{R}_{<0} \times \mathbb{R}_{<0}) = \varnothing$.
\end{proof}
\begin{Lemma}\label{true}
 If $R_{s}$ doesn't contain $(0,0)$ for any $s \in (0,1)$, then $L_{s_{1},s_{2}} \not\simeq c''$ in $X$.
\end{Lemma}
\begin{proof}
 Since neither $R_{s_{1}}$ nor $R_{s_{2}}$ contain the origin, $L_{s_{1},s_{2}}$ has a well-defined winding number around $(0,0)$. We show that the winding number $\eta(L_{s_{1},s_{2}})$ of $L_{s_{1},s_{2}}$ around $(0,0)$ is 1. As $\eta(c'') = 0$, and as the winding number of a curve is a homotopy invariant, demonstrating this shall be sufficient for this proof.
 
 For any $t \in [0,1]$, let $\theta(t)$ be equal to the angle of $L_{s_{1},s_{2}}(t)$ relative to the $x$-axis. To compute $\eta(L_{s_{1},s_{2}})$, we first find $\theta(\frac{1}{2}) - \theta(0)$ and $\theta(1) - \theta(\frac{1}{2})$.
 
\begin{itemize}
    \item $\theta(\frac{1}{2}) - \theta(0)$: Since $L_{s_{1},s_{2}}\mid_{[0,\frac{1}{2}]} = R_{s_{1}} \subset \mathbb{R}_{\leq 0} \times \mathbb{R}$ (from Lemma~\ref{negativex}), $\| \theta(\frac{1}{2}) - \theta(0) \| \not\geq 1$, because that would require $L_{s_{1},s_{2}}$ to pass through all 4 quadrants.
    
    Then, as $R_{s_{2}}$ goes from $(-1,0)$ to $(0,-1)$, $\theta(\frac{1}{2}) - \theta(0)$ can either be $\frac{\pi}{2}$ or $-\frac{3\pi}{2}$. However $\theta(\frac{1}{2}) - \theta(0) = -\frac{3\pi}{2}$, only if $L_{s_{1},s_{2}}\mid_{[0,\frac{1}{2}]}$ passes through $\mathbb{R}_{>0} \times \mathbb{R}$, in contradiction to Lemma~\ref{negativex}. Therefore, $\theta(\frac{1}{2}) - \theta(0) \neq -\frac{3\pi}{2}$, and thus $\theta(\frac{1}{2}) - \theta(0) = \frac{\pi}{2}$. See Figure~\ref{clockwise} for a diagram illustrating this.

\begin{figure}
    \begin{subfigure}{0.43\textwidth}
\resizebox{0.9\linewidth}{!}{\input{Clockwise.tex}} 
\caption{$\theta(\frac{1}{2}) - \theta(0) = -\frac{3\pi}{2}$ only if $L_{s_{1},s_{2}}\mid_{[0,\frac{1}{2}]}$ passes through $\mathbb{R}_{>0} \times \mathbb{R}$.}
\label{clockwise}
\end{subfigure}
\hfill
    \begin{subfigure}{0.45\textwidth}
\resizebox{0.9\linewidth}{!}{\input{anticlockwise.tex}} 
\caption{$\theta(1) - \theta(\frac{1}{2}) = -\frac{\pi}{2}$ only if $L_{s_{1},s_{2}}\mid_{[\frac{1}{2},1]}$ passes through $\mathbb{R}_{<0} \times \mathbb{R}_{<0}$.}\hfill
\label{anticlockwise}
\end{subfigure}
\hfill
\end{figure}
    
    \item $\theta(1) - \theta(\frac{1}{2})$: Since $L_{s_{1},s_{2}}\mid_{[\frac{1}{2},1]} = \bar{R}_{s_{2}} \not\subset \mathbb{R}_{< 0} \times \mathbb{R}_{<0}$ (from Lemma~\ref{positivequads}), $\| \theta(1) - \theta(\frac{1}{2}) \| \not\geq 1$, because that would again require $L_{s_{1},s_{2}}\mid_{[0,\frac{1}{2}]}$ to pass through all 4 quadrants.
    
    Then, as $L_{s_{1},s_{2}}\mid_{[0,\frac{1}{2}]}$ goes from $(0,-1)$ to $(-1,0)$, $\theta(1) - \theta(\frac{1}{2})$ can either be $\frac{\pi}{2}$ or $-\frac{3\pi}{2}$. However, $\theta(1) - \theta(\frac{1}{2}) = -\frac{\pi}{2}$ only if $L_{s_{1},s_{2}}\mid_{[\frac{1}{2},1]}$ travels through $\mathbb{R}_{<0} \times \mathbb{R}_{<0}$, in contradiction to Lemma~\ref{positivequads}. Therefore, $\theta(1) - \theta(\frac{1}{2}) \neq -\frac{\pi}{2}$, and thus $\theta(1) - \theta(\frac{1}{2}) = \frac{3\pi}{2}$. See Figure~\ref{anticlockwise} for a diagram illustrating this.
\end{itemize}
Hence,

$$\eta(L_{s_{1},s_{2}}) = \frac{\theta(1) - \theta(0)}{2\pi} = \frac{(\theta(1) - \theta(\frac{1}{2})) + (\theta(\frac{1}{2}) - \theta(0))}{2\pi} = \frac{\frac{\pi}{2} + \frac{3\pi}{2}}{2\pi} = 1.$$

Then, as discussed above, it follows that $L_{s_{1},s_{2}} \not\simeq c''$ in $X'$.
\end{proof}
We now conclude with the proof of Theorem~\ref{General}.
\begin{proof}[Proof of Theorem~\ref{General}]
 If $R_{s}$ doesn't contains $(0,0)$ for any $s\in (0,1)$, then Lemmata~\ref{false} and~\ref{true} tell us respectively that $L_{s_{1},s_{2}} \simeq c''$ in $X'$ and that $L_{s_{1},s_{2}} \not\simeq c''$, which is a contradiction. Therefore, our assumption for these lemmata is false, so there exists some $s'$ and $t'$ such that $R_{s'}(t') = (0,0)$. This gives us two points---namely, $\gamma(s')$ and $\gamma(t')$---that, by the construction of $R_{s}$, inscribe an equilateral triangle such that one of its vertices is $o$. 
\end{proof}

\section*{Acknowledgements}

We would like to thank Rachana Madhukara for helpful comments.

\bibliographystyle{alpha}
\bibliography{inscribed}

\begin{thebibliography}{Emc16}

\bibitem[AK13]{karasev2009inscribing}
Arseniy Akopyan and Roman Karasev.
\newblock Inscribing a regular octahedron into polytopes, 2013.

\bibitem[Emc16]{10.2307/2370541}
Arnold Emch.
\newblock On some properties of the medians of closed continuous curves formed
  by analytic arcs.
\newblock {\em Amer. J. Math.}, 38(1):6--18, 1916.

\bibitem[Mak16]{article}
V.~V. Makeev.
\newblock Planar sections of three-dimensional cylinders.
\newblock {\em Vestnik St. Petersburg Univ. Math.}, 49(4):359--360, 2016.

\bibitem[Mat14]{Matschke2014ASO}
Benjamin Matschke.
\newblock A survey on the square peg problem.
\newblock {\em Notices Amer. Math. Soc.}, 61(4):346--352, 2014.

\bibitem[Mey80]{Meyerson1980}
Mark~D. Meyerson.
\newblock Equilateral triangles and continuous curves.
\newblock {\em Fund. Math.}, 110(1):1--9, 1980.

\bibitem[Mey81]{meyerson1981balancing}
Mark~D. Meyerson.
\newblock Balancing acts.
\newblock {\em Topology Proc.}, 6(1):59--75 (1982), 1981.

\bibitem[Nie92]{Nielsen1992}
Mark~J. Nielsen.
\newblock Triangles inscribed in simple closed curves.
\newblock {\em Geom. Dedicata}, 43(3):291--297, 1992.

\bibitem[NW95]{Nielsen1995}
Mark~J. Nielsen and S.~E. Wright.
\newblock Rectangles inscribed in symmetric continua.
\newblock {\em Geom. Dedicata}, 56(3):285--297, 1995.

\bibitem[Pak10]{pak2010lectures}
Igor Pak.
\newblock Lectures on discrete and polyhedral geometry.
\newblock \url{http://www.math.ucla.edu/~ pak/book.htm}, 2010.

\bibitem[Str89]{stromquist_1989}
Walter Stromquist.
\newblock Inscribed squares and square-like quadrilaterals in closed curves.
\newblock {\em Mathematika}, 36(2):187--197 (1990), 1989.

\end{thebibliography}

\end{document}